\documentclass[a4paper,11pt]{article}
\usepackage{mathbbold}
\usepackage{amsfonts}
\usepackage{bbding}
\usepackage{amssymb}
\usepackage{mathrsfs}
\usepackage{float}
\usepackage{amsmath}
\usepackage{amsfonts,amssymb}
\usepackage{graphicx,color}
\usepackage{psfig}
\usepackage{mathrsfs}
\usepackage{graphicx}
\usepackage{amsthm}
\usepackage{dsfont}
 \catcode`\@=11 \catcode`\@=12

\newtheorem{Theorem}{Theorem}[section]
\newtheorem{Definition}{Definition}[section]
\newtheorem{Lemma}{Lemma}[section]
\newtheorem{Example}{Example}[section]

\newtheorem{Remark}{Remark}[section]

\catcode`\@=11 \@addtoreset{equation}{section}

\title{Two-point boundary value problems and exact controllability for several kinds of linear and nonlinear wave
equations}
\author{De-Xing Kong\footnote{Department of Mathematics, Zhejiang University,
Hangzhou 310027, China.}  $\quad$ and $\quad$ Qing-You
Sun\footnote{Center of Mathematical Sciences, Zhejiang University,
Hangzhou 310027, China.}\\}

\date{ }
\begin{document}
\maketitle

\noindent{\bf Abstract}: In this paper we introduce some new
concepts for second-order hyperbolic equations: two-point boundary
value problem, global exact controllability and exact
controllability. For several kinds of important linear and nonlinear
wave equations arising from physics and geometry, we prove the
existence of smooth solutions of the two-point boundary value
problems and show the global exact controllability of these wave
equations. In particular, we investigate the two-point boundary
value problem for one-dimensional wave equation defined on a closed
curve and prove the existence of smooth solution which implies the
exact controllability of this kind of wave equation. Furthermore,
based on this, we study the two-point boundary value problems for
the wave equation defined on a strip with Dirichlet or Neumann
boundary conditions and show that the equation still possesses the
exact controllability in these cases. Finally, as an application, we
introduce the hyperbolic curvature flow and obtain a result
analogous to the well-known theorem of Gage and Hamilton for the
curvature flow of plane curves.

\vskip 10mm


\vskip 6mm


\newpage
\baselineskip=7mm

\section{Introduction}\label{sec:1}

Consider the following seconder-order hyperbolic partial
differential equation
\begin{equation} \label{eq:1-1.0}
\mathscr{P}(t,x,u,Du,D^2u)=0,
\end{equation}
where $t$ is the time variable, $x=(x_1,\cdots x_n)$ stand for the
spacial variables, $u=u(t,x)$ is the unknown function, $\mathscr{P}$
is a given smooth function of the independent of variables
$t,x_1,\cdots,x_n$, the unknown function $u$, the first-order
partial derivatives $Du=(u_t,u_{x_1},\cdots,u_{x_n})$ and the
second-order partial derivatives $D^2u=(u_{tt},u_{tx_1},\cdots)$.
Since we only consider the hyperbolic case, the initial data
associated with the equation \eqref{eq:1-1.0} read
\begin{equation}\label{eq:1-1.00}
t=0:\;\;u=u_0(x),\quad u_t=u_1(x),
\end{equation}
where $u_0(x)$ and $u_1(x)$ are two given functions which stand for
the initial position and the initial velocity, respectively. The
equation \eqref{eq:1-1.0} and the initial data \eqref{eq:1-1.00}
constitute the famous Cauchy problem. Another important problem
related to \eqref{eq:1-1.0} is the following so-called {\it
two-point boundary value problem} for the equation \eqref{eq:1-1.0}:

\vskip 3mm

\noindent{\bf Two-point Boundary Value Problem (TBVP):} $\;${\it
Given two suitable smooth functions $u_0(x),u_T(x)$ and a positive
constant $T$, can we find a $C^2$-smooth function $u=u(t,x)$ defined
on the strip $[0,T]\times\mathds{R}^{n}$ such that the function
$u=u(t,x)$ satisfies the equation \eqref{eq:1-1.0} on the domain
$[0,T]\times\mathds{R}^{n}$, the initial condition
\begin{equation}\label{eq:1-1.01}
u(0,x)=u_0(x),\quad\forall\;x\in \mathds{R}^{n}
\end{equation}
and the terminal condition
\begin{equation}\label{eq:1-1.02}
u(T,x)=u_T(x),\quad\forall\;x\in \mathds{R}^{n}?
\end{equation}}

\noindent{\bf Another Statement of TBVP:} $\;${\it Given two
suitable smooth functions $u_0(x),u_T(x)$ and a positive constant
$T$, can we find an initial velocity $u_t(0,x)=u_1(x)$ such that the
Cauchy problem \eqref{eq:1-1.0}-\eqref{eq:1-1.00} has a solution
$u=u(t,x)\in C^2([0,T]\times\mathds{R}^{n})$ which satisfies the
terminal condition \eqref{eq:1-1.02}?}

The system \eqref{eq:1-1.0} together with
\eqref{eq:1-1.01}-\eqref{eq:1-1.02} can be viewed as a distributed
parameter control system when the initial velocity function $u_1(x)$
is considered as a control input.

We now give the following definition.

\begin{Definition}
For any given $T>0$, if the TBVP \eqref{eq:1-1.0},
\eqref{eq:1-1.01}-\eqref{eq:1-1.02} has a $C^2$ solution on the
strip $[0,T]\times\mathds{R}^{n}$, then the equation
\eqref{eq:1-1.0} is called to possess the global exact
controllability; If the TBVP \eqref{eq:1-1.0},
\eqref{eq:1-1.01}-\eqref{eq:1-1.02} admits a $C^2$ solution on the
strip $[0,T]\times\mathds{R}^{n}$ for some given $T>0$, then we say
that the equation \eqref{eq:1-1.0} possesses the exact
controllability.
\end{Definition}

\begin{Remark} \label{rem:1-1.0} For the system of
seconder-order hyperbolic partial differential equations, we have
similar concepts and definitions; For higher-order hyperbolic
partial differential equations, we have a similar discussion.
\end{Remark}

The wave equations play an important role in both theoretical and
applied fields, they include two classes: linear wave equations and
nonlinear wave equations. The classical wave equation is an
important second-order linear partial differential equation of
waves, such as sound waves, light waves and water waves. It arises
in fields such as acoustics, electromagnetics, fluid dynamics, and
general relativity. Historically, the problem of a vibrating string
such as that of a musical instrument was studied by Jean le Rond
d'Alembert, Leonhard Euler, Daniel Bernoulli, and Joseph-Louis
Lagrange.

The wave equation is the prototypical example of a hyperbolic
partial differential equation. In its simplest form, the wave
equation refers to a unknown scalar function $y=y(t,x)$ which
satisfies
\begin{equation} \label{eq:1-1.1}
y_{tt}-c^2\Delta y=0,
\end{equation}
where ${\displaystyle\Delta=\sum^n_{i=1}\frac{\partial^2}{\partial
x_i^2}}$ denotes the Laplacian and $c$ is a fixed constant which
stands for the propagation speed of the wave. One of aims of the
present is to investigate the TBVP for \eqref{eq:1-1.1}, more
precisely we study the following TBVP:

\vskip 3mm

\noindent{\bf TBVP for \eqref{eq:1-1.1}:} $\;${\it Given two
functions $f(x),g(x)\in C^{[n/2]+2}(\mathds{R}^{n})$ and a positive
constant $T$, can we find a $C^2$-smooth function $y=y(t,x)$ defined
on the strip $[0,T]\times\mathds{R}^{n}$ such that the function
$y=y(t,x)$ satisfies the equation \eqref{eq:1-1.1} on the domain
$[0,T]\times\mathds{R}^{n}$, the initial condition
\begin{equation}\label{eq:1-1.2}
y(0,x)=f(x),\quad\forall\;x\in \mathds{R}^{n}
\end{equation}
and the terminal condition
\begin{equation}\label{eq:1-1.3}
y(T,x)=g(x),\quad\forall\;x\in \mathds{R}^{n}?
\end{equation}}

\noindent{\bf Another Statement:} $\;${\it Given two functions
$f(x),g(x)\in C^{[n/2]+2}(\mathds{R}^{n})$ and a positive constant
$T$, can we find an initial velocity $v=v(x)$ such that the Cauchy
problem for the wave equation \eqref{eq:1-1.1} with the initial data
\begin{equation}\label{eq:1-1.4}
t=0:\;\;y=f(x),\quad y_t=v(x)
\end{equation}
has a solution $y=(t,x)\in C^2([0,T]\times\mathds{R}^{n})$ which
satisfies the terminal condition \eqref{eq:1-1.3}?}

In this paper we shall show the global exact controllability of the
equation \eqref{eq:1-1.1} and some nonlinear wave equations arising
from geometry and the theory of relativistic strings.

It is well-known that there are many deep and beautiful results on
the TBVP for ordinary differential equations, however, according to
the authors' knowledge, few of results on the TBVP for hyperbolic
equations, even for (linear or nonlinear) wave equations have been
known. Therefore, we can say that the result presented in this paper
is the first result on this research topic.

On the other hand, the study on {\it boundary} control problems for
hyperbolic systems was initiated by D.L. Russell in the 1960s. In
\cite{Ru1}, using the characteristic method, he showed that a class
of $n\times n$ first order linear hyperbolic systems is exactly
boundary controllable. This work led to an intensive investigation
of controllability and stabilization of linear hyperbolic systems
for more than 30 years. The literature pertaining to this study is
now absolutely enormous, we refer to two excellent review papers
Russell \cite{Ru2} and Lions \cite{lions}. However, while it may be
fair to say that the study of boundary control of linear hyperbolic
systems is now nearly complete, the study of nonlinear hyperbolic
systems is still vastly open.  Up to now, some results on the exact
boundary controllability of (abstract) semilinear wave equations
have been obtained (see \cite{chen}, \cite{che}, \cite{fa},
\cite{LT1}-\cite{LT2}, \cite{zu1}-\cite{zu3} and references cited
therein). As for boundary control of quasilinear hyperbolic systems,
there have been few results so far. Motivated by Ruessell's work,
Cirin\`a \cite{ci} studied boundary control of general quasilinear
hyperbolic systems. Using a different approach from that of Russell,
he proved that the system is {\it locally} exactly boundary
controllable in the sense that the $C^1$ norms of both initial and
terminal states are required to be small.

All the above results are obtained under the assumption that the
initial and terminal states are smooth, and they are discussed in
the framework of classical solutions. For the global exact
controllability of system (1.1) in the case that the initial and
terminal states maybe contain discontinuity points of the first
kind, up to now only a few of results have been known. In Kong
\cite{kong1}, the author investigated the global exact boundary
controllability of $2\times 2$ quasilinear hyperbolic system of
conservation laws with linearly degenerate characteristics and
proved that the system with nonlinear boundary conditions is
globally exactly boundary controllable in the class of piecewise
$C^1$ functions. Later, by a new constructive method, Kong and Yao
reproved the global exact boundary controllability of a class of
quasilinear hyperbolic systems of conservation laws with linearly
degenerate characteristics, shown that the system with nonlinear
boundary conditions is globally exactly boundary controllable in the
class of piecewise $C^1$ functions, in particular, gave the optimal
control time of the system (see \cite{kong2}).

Here we would like to remark that the TBVP and the boundary control
problems are essentially different two kinds of problems. Both of
them play an important role in both theoretical and applied aspects.

We now state the first result in this paper.
\begin{Theorem} \label{thm:1-1.1}
The TBVP \eqref{eq:1-1.1}-\eqref{eq:1-1.3} admits a $C^2$-smooth
solution $y = y(t,x)$ defined on the strip
$[0,T]\times\mathds{R}^n$.
\end{Theorem}

Theorem \ref{thm:1-1.1} implies that the wave equation
\eqref{eq:1-1.1} possesses the global exact controllability.

\begin{Remark} \label{rem:1-1.1}
The solution of the TBVP of \eqref{eq:1-1.1}-\eqref{eq:1-1.3} does
not possess the uniqueness (see Remarks \ref{rem:2.2} and
\ref{rem:3.2} for the details).
\end{Remark}

\begin{Remark} \label{rem:1-1.2}
Theorem \ref{thm:1-1.1} can be generalized to the case of
inhomogeneous wave equations, i.e.,
\begin{equation} \label{eq:1-1.5}
y_{tt}-c^2\Delta y=F(t,x),
\end{equation}
where $F(t,x)$ is a given function which stands for the source term
of the system.
\end{Remark}

\begin{Remark} \label{rem:1-1.3}
We have similar results for some nonlinear wave equations including
a wave map equation arising from geometry and the equations for the
motion of relativistic strings in the Minkowski space-time
$\mathds{R}^{1+n}$ (see Section \ref{sec:7} for the details). These
nonlinear wave equations play an important role in both mathematics
and physics.
\end{Remark}

However, some problems arising from engineering, control theory etc.
can be reduced to the TBVP for wave equations defined on a closed
curve, say, a circle. The typical example is the vibration of a
closed elastic string. In this case, since the wave equation is
defined on a closed curve and then the solution must possess the
periodicity, the method used in the proof of Theorem \ref{thm:1-1.1}
will no longer work. It needs some new ideas and new technologies to
study such a kind of problems.

In this paper we also investigate this kind of problems mentioned
above. For simplicity, we consider the following TBVP\footnote{In
fact, by scaling, any wave equation $y_{tt}-c^2y_{\theta\theta}=0$
with the propagation speed $c$ can be reduced to the wave equation
in \eqref{eq:1-2.5}. Therefore, in this paper, without loss of
generality, we only consider the wave equation with the propagation
speed $c=1$.}

\begin{equation} \label{eq:1-2.5}
\begin{cases}
y_{tt}-y_{\theta\theta}=0,\quad \forall\;(t,\theta)\in [0,T]\times \mathds{R},\\
y(0,\theta)=f(\theta),\quad \forall\;\theta\in \mathds{R},\\
y(T,\theta)=g(\theta),\quad \forall\;\theta\in \mathds{R},
\end{cases}
\end{equation}
where $y=y(t,\theta)$ is the unknown function of the variables $t$
and $\theta$, $T$ is a given positive constant, $f(\theta)$ and
$g(\theta)$ are two given periodic $C^3$ functions of $\theta\in
\mathds{R}$, say, the period is $L$, in which $L$ is a positive real
number. The second result in the present paper is the following
theorem.

\begin{Theorem} \label{thm:1-2.1}
The TBVP \eqref{eq:1-2.5} admits a global $L$-periodic $C^2$
solution $y = y(t,\theta)$ defined on the domain
$[0,T]\times\mathds{R}$, provided that $\displaystyle\frac{T}{L}$ is
a rational number with $\displaystyle\frac{2T}{L}\not\in
\mathds{N}$.
\end{Theorem}

\begin{Remark} \label{rem:1-2.1}
Theorem \ref{thm:1-2.1} implies that the wave equation in
\eqref{eq:1-2.5} possesses the exact controllability. On the other
hand, in general, the solution of the TBVP \eqref{eq:1-2.5} is not
unique (see the proof of Theorem \ref{thm:1-2.1} in Section
\ref{sec:4} for the details).
\end{Remark}

\begin{Remark} \label{rem:1-2.2}
In Theorem \ref{thm:1-2.1}, if $\displaystyle\frac{2T}{L}\in
\mathds{N}$, then there exists a relationship between $f(\theta)$
and $g(\theta)$. This means that $f(\theta)$ and $g(\theta)$ can not
be given arbitrarily. See Remark \ref{rem:4.2} for the details.
\end{Remark}

As a consequence, we consider the following TBVP defined on a circle
\begin{equation} \label{eq:1-2.6}
\begin{cases}
y_{tt}-y_{\theta\theta}=0,\quad \forall\;(t,\theta)\in [0,T]\times \mathbb{S}^1,\\
y(0,\theta)=f(\theta),\quad \forall\;\theta\in \mathbb{S}^1,\\
y(T,\theta)=g(\theta),\quad \forall\;\theta\in \mathbb{S}^1,
\end{cases}
\end{equation}
where $\mathbb{S}^1$ stands for the unit circle, $f(\theta)$ and
$g(\theta)$ are two given $C^3$ functions of
$\theta\in\mathbb{S}^1$. By Theorem \ref{thm:1-2.1}, we have

\begin{Theorem} \label{thm:1-2.2}
The TBVP \eqref{eq:1-2.6} admits a global $C^2$ solution $y =
y(t,\theta)$ defined on the cylinder $[0,T]\times\mathbb{S}^1$,
provided that $\displaystyle\frac{T}{2\pi}$ is a rational number and
$\displaystyle\frac{T}{\pi}\not\in \mathds{N}$.
\end{Theorem}

\begin{Remark} \label{rem:1-2.3}
Theorem \ref{thm:1-2.2} implies that the wave equation defined on a
circle possesses the exact controllability.
\end{Remark}

\begin{Remark} \label{rem:1-2.4}
For the $(1+n)$-dimensional wave equation \eqref{eq:1-1.1}, we have
similar results, provided that the initial/terminal data $f(x)$ and
$g(x)$ in \eqref{eq:1-1.2}-\eqref{eq:1-1.3} are $C^{[n/2]+3}$ smooth
functions and are periodic in $r=\sqrt{x_1^2+\cdots +x_n^2}$. In
fact, in the present situation, using Theorem \ref{thm:1-2.1}, by a
way similar to the proof of Theorem \ref{thm:1-1.1} we can prove
that the TBVP \eqref{eq:1-1.1}-\eqref{eq:1-1.3} admits a global
$C^2$ solution $y = y(t,x)$ defined on the strip $[0,T]\times
\mathds{R}^n$, provided that $\displaystyle\frac{T}{L}$ is a
rational number with $\displaystyle\frac{2T}{L}\not\in \mathds{N}$,
where $L$ is the period of $f(x),\;g(x)$ with respect to $r$. In
other words, in this case the equation \eqref{eq:1-1.1} still
possesses the exact controllability.
\end{Remark}

Based on the results mentioned above, we further study the two-point
boundary value problems for the wave equation defined on a strip
with Dirichlet or Neumann boundary conditions and show that the
equation still possesses the exact controllability in these cases.
Finally, as an application of the results mentioned above, we
introduce the hyperbolic curvature flow and obtain a result
analogous to the well-known theorem of Gage and Hamilton \cite{gh}
for the curvature flow of plane curves.

This paper is organized as follows. Section \ref{sec:2} is devoted
to the study on the global exact controllability of one-dimensional
wave equation. Based on Section \ref{sec:2}, in Section \ref{sec:3}
we investigate the global exact controllability of linear wave
equations in several space variables and we prove Theorem
\ref{thm:1-1.1}. In Section \ref{sec:4}, we give the proof of
Theorem \ref{thm:1-2.1} and Theorem \ref{thm:1-2.2}, respectively.
As some applications of Theorem \ref{thm:1-2.1}, in Sections
\ref{sec:5} and \ref{sec:6} we study the two-point boundary value
problems for the wave equation defined on a strip with (homogeneous
or inhomogeneous) Dirichlet or Neumann boundary conditions and show
that the equation still possesses the exact controllability in these
cases. The global exact controllability of some nonlinear wave
equations arising from geometry and physics has been investigated in
Section \ref{sec:7}. In Section \ref{sec:8} we introduce the
hyperbolic curvature flow and prove a result analogous to the one
shown by Gage and Hamilton in \cite{gh} for curvature flow of plane
curves. In Section \ref{sec:9} we give a summary and some
discussions and then present several open problems.

\section{One-dimensional wave equation}\label{sec:2}

This section concerns the global exact controllability of
one-dimensional wave equation, which is a basis of the present
paper.

Consider the following TBVP for one-dimensional wave equation
\begin{equation} \label{eq:2.1}
\begin{cases}
y_{tt} - y_{xx} = 0,\\
y(0,x) = f(x),\\
y(T,x) = g(x),
\end{cases}
\end{equation}
where $T$ is an arbitrary fixed positive constant, $f(x)$ and $g(x)$
are two given functions of $x\in\mathds{R}$. We have the following
theorem which is a special case of Theorem \ref{thm:1-1.1} but
fundamental in this paper.
\begin{Theorem} \label{thm:2.1}
Suppose that $f(x)$ and $g(x)$ are two given $C^2$-smooth functions
of $x\in\mathds{R}$. Then the TBVP \eqref{eq:2.1} admits a
$C^2$-smooth solution $y = y(t,x)$ defined on the strip
$[0,T]\times\mathds{R}$.
\end{Theorem}

\begin{proof}
Introduce
\begin{equation} \label{eq:2.2}
\tilde{y}(t,x) = y(t,x) - \displaystyle\frac{f(x-t)+f(x+t)}{2},
\end{equation}
and
\begin{equation} \label{eq:2.3}
\tilde{f}(x) = g(x) - \displaystyle\frac{f(x-T)+f(x+T)}{2}.
\end{equation}
Obviously, $\tilde{f}(x)$ is a $C^2$-smooth function of
$x\in\mathds{R}$. By means of $\tilde{y}(t,x)$ and $\tilde{f}(x)$,
the TBVP \eqref{eq:2.1} can be equivalently rewritten as
\begin{equation} \label{eq:2.4}
\begin{cases}
\tilde{y}_{tt} - \tilde{y}_{xx} = 0,\\
\tilde{y}(0,x) = 0,\\
\tilde{y}(T,x) = \tilde{f}(x),
\end{cases}
\end{equation}
Therefore, in order to prove Theorem \ref{thm:2.1}, it suffices to
show that the TBVP \eqref{eq:2.4} has a $C^2$-smooth solution
$\tilde{y} = \tilde{y}(t,x)$ defined on the strip
$[0,T]\times\mathds{R}$.

By d'Alembert formula, the solution of the following Cauchy problem
\begin{equation} \label{eq:2.5}
\begin{cases}
\tilde{y}_{tt} - \tilde{y}_{xx} = 0,\\
t=0:\;\;\tilde{y} = 0,\quad \tilde{y}_t = v(x)
\end{cases}
\end{equation}
reads
\begin{equation} \label{eq:2.6}
\tilde{y}(t,x) =
\displaystyle\frac{1}{2}\int_{x-t}^{x+t}{v(\tau)}d\tau,
\end{equation}
where $v(x)$ is a $C^1$-smooth function to be determined, which
stands for the initial velocity.

We next show that there indeed exists an initial velocity $v(x)$
such that the solution $\tilde{y} = \tilde{y}(t,x)$ of the Cauchy
problem \eqref{eq:2.5}, defined by \eqref{eq:2.6}, satisfies the
terminal condition in the TBVP \eqref{eq:2.4}, i.e.,
\begin{equation} \label{eq:2.7}
\tilde{y}(T,x) = \tilde{f}(x).\end{equation}

To do so, we construct a $C^1$ function $u(x)$ on $[-T, T]$ which
satisfies
\begin{equation} \label{eq:2.8}
\int_{-T}^T u(\tau)d\tau = 2\tilde{f}(0), \quad  u(T) - u(-T) =
2\tilde{f}'(0) \quad \text{and}\quad u'_-(T) - u'_+(-T) =
2\tilde{f}''(0).
\end{equation}
Define
\begin{equation} \label{eq:2.9}
v(x) =
\begin{cases}
u(x-2NT) + 2\displaystyle\sum_{i=1}^N\tilde{f}'\big(x-(2i-1)T\big), \quad\forall\; x\geq 0,\\
u(x+2NT) - 2\displaystyle\sum_{i=1}^N\tilde{f}'\big(x+(2i-1)T\big),
\quad\forall\; x<0,
\end{cases}
\end{equation}
where $N$ is given by
$$N=\left[\displaystyle\frac{|x|+T}{2T}\right].$$
In \eqref{eq:2.9}, the terms including the summation disappear in
the case of $N=0$.

We claim that the function $v(x)$ defined by \eqref{eq:2.9} is
$C^1$-smooth.

In what follows, we distinguish two cases to show this fact.

{\bf Case A: $x\geq 0$.} $\quad$ In this case, for every $x\neq
(2N-1)T\;\; (N\in\mathds{N})$, by \eqref{eq:2.9} it is easy to check
that $v(x)$ is $C^1$-smooth at such a point $x$.

When $x=(2N-1)T$, it follows from \eqref{eq:2.9} that
\begin{equation} \label{eq:2.10}
\lim_{\alpha\rightarrow x^+}v(\alpha) = v(x) = u(-T) +
2\sum_{i=0}^{N-1}\tilde{f}'(2iT)
\end{equation}
and
\begin{equation} \label{eq:2.11}
\lim_{\alpha\rightarrow x^-}v(\alpha)  = \lim_{\alpha\rightarrow
x^-}u\big(\alpha-2(N-1)T\big) +
2\sum_{i=1}^{N-1}\lim_{\alpha\rightarrow
x^-}\tilde{f}'\big(\alpha-(2i-1)T\big) = u(T) +
2\sum_{i=1}^{N-1}\tilde{f}'(2iT).
\end{equation}
Combining \eqref{eq:2.10} and \eqref{eq:2.11} gives
\begin{equation} \label{eq:2.12}
v(x)- \lim_{\alpha\rightarrow x^-}v(\alpha)= u(-T) - u(T) +
2\tilde{f}'(0).
\end{equation}
Noting the second equation in \eqref{eq:2.8} yields the
continuity of $v(x)$ for all $x\geq 0$.\\

On the other hand, it follows from \eqref{eq:2.9} that
\begin{equation} \label{eq:2.13}
v'_-(x) = u'_-\big(x-2(N-1)T\big) +
2\sum_{i=1}^{N-1}\tilde{f}''\big(x-(2i-1)T\big)= u'_-(T) +
2\sum_{i=1}^{N-1}\tilde{f}''\big((2N-2i)T\big)
\end{equation}
and
\begin{equation} \label{eq:2.14}
v'_+(x) = u'_+(x-2NT) + 2\sum_{i=1}^N\tilde{f}''\big(x-(2i-1)T\big)=
u'_+(-T) + 2\sum_{i=1}^N\tilde{f}''\big((2N-2i)T\big).
\end{equation}
Combining \eqref{eq:2.13} and \eqref{eq:2.14} gives
\begin{equation} \label{eq:2.15}
v'_+(x) - v'_-(x) = u'_+(-T) - u'_-(T) + 2\tilde{f}''(0).
\end{equation}
Noting the third equation in \eqref{eq:2.8}, we have
\begin{equation} \label{eq:2.16}
v'_+(x) = v'_-(x).
\end{equation}
Summarizing the above argument yields that the function $v(x)$
defined by \eqref{eq:2.9} is $C^1$-smooth for all $x\geq 0$.

{\bf Case B: $x< 0$.} $\quad$ Similarly, we can prove that the
function $v(x)$ is $C^1$-smooth in the present situation.

Combining Cases A and B, we have shown that the function $v(x)$
defined by \eqref{eq:2.9} is a $C^1$-smooth function of
$x\in\mathds{R}$, and then the solution $\tilde{y}=\tilde{y}(t,x)$,
defined by \eqref{eq:2.6}, of the Cauchy problem \eqref{eq:2.5} is
$C^2$-smooth on the whole upper plane
$\mathds{R}^+\times\mathds{R}$.

We now claim that, for the initial velocity $v(x)$ defined by
\eqref{eq:2.9}, the solution $\tilde{y}=\tilde{y}(t,x)$ defined by
\eqref{eq:2.6} satisfies the terminal condition \eqref{eq:2.7}.

In fact, we verify this statement by distinguishing the following
two cases:

{\bf Case 1: $x\geq 0$.} $\quad$ In the present situation, it holds
that
$$x\leq 2NT+T \leq x+2T.$$
Thus it follows from \eqref{eq:2.9} that
\begin{align} \label{eq:2.17}
\displaystyle\frac{1}{2}\int_{x}^{x+2T}{v(\tau)}d\tau
&=\; \frac{1}{2}\left[\int_{x}^{2NT+T}{v(\tau)}d\tau + \int_{2NT+T}^{x+2T}{v(\tau)}d\tau\right] \nonumber \\
&=\; \int_{x}^{2NT+T}\left[\frac{1}{2}u(\tau-2NT) +
\sum_{i=1}^N\tilde{f}'\big(\tau-(2i-1)T\big)\right]d\tau +\nonumber \\
&\quad\; \int_{2NT+T}^{x+2T}\left[\frac{1}{2}u\big(\tau-2(N+1)T\big)
+
\sum_{i=1}^{N+1}\tilde{f}'\big(\tau-(2i-1)T\big)\right]d\tau \nonumber \\
&=\;\frac{1}{2}\left[\int_{x}^{2NT+T}u(\tau)d\tau +
\int_{2NT+T}^{x+2T}u(\tau)d\tau \right]+ \nonumber \\
&\;\quad \sum_{i=1}^N\left[\tilde{f}\big((2N-2i+2)T\big) -
\tilde{f}\big(x-(2i-1)T\big)\right] +\nonumber \\
&\;\quad \sum_{i=1}^{N+1}\left[\tilde{f}\big(x-(2i-3)T\big) -
\tilde{f}\big((2N-2i+2)T\big)\right] \nonumber \\
&=\;\frac{1}{2}\left[\int_{x-2NT}^{T}u(\tau)d\tau +
\int_{-T}^{x-2NT}u(\tau)d\tau \right] +\nonumber \\
&\;\quad \sum_{i=1}^N\left[\tilde{f}\big(x-(2i-3)T\big) -
\tilde{f}\big(x-(2i-1)T\big)\right] +\nonumber \\
&\;\quad  \tilde{f}\big(x-(2N-1)T\big)- \tilde{f}(0) \nonumber \\
&= \;\frac{1}{2}\int_{-T}^T{u(\tau)}d\tau + \tilde{f}(x+T) -
\tilde{f}(0).
\end{align}
Noting the first condition of \eqref{eq:2.8}, we obtain
\begin{equation} \label{eq:2.18}
\displaystyle\frac{1}{2}\int_{x}^{x+2T}{v(\tau)}d\tau =
\tilde{f}(x+T).
\end{equation}
This leads to
\begin{equation} \label{eq:2.19}
\tilde{y}(T,x) =
\displaystyle\frac{1}{2}\int_{x-T}^{x+T}{v(\tau)}d\tau =
\tilde{f}(x).
\end{equation}
This is the desired terminal condition \eqref{eq:2.7} for the case
of $x\geq 0$.

{\bf Case 2: $x< 0$.} $\quad$ In a similar manner, we can prove the
terminal condition \eqref{eq:2.7} holds for all $x<0$.

The above discussion shows that $\tilde{y}=\tilde{y}(t,x)$ defined
by \eqref{eq:2.6} is a $C^2$-smooth solution of the TBVP
\eqref{eq:2.4} defined on the strip $[0,T]\times\mathds{R}$. This
proves Theorem \ref{thm:2.1}.
\end{proof}

\begin{Remark} \label{rem:2.1}
In order to illustrate that it is easy to construct the function $u$
satisfying \eqref{eq:2.8}, in this remark we present two examples.
The first example reads
\begin{equation} \label{eq:2.20}
u=\frac{\tilde{f}''(0)}{2T}x^2+\frac{\tilde{f}'(0)}{T}x+\frac{\tilde{f}(0)}{T}
-\frac{\tilde{f}''(0)T}{6}.
\end{equation}
It is easy to check that the function $u$ defined by \eqref{eq:2.20}
satisfies all conditions in \eqref{eq:2.8}. The second example is
\begin{equation} \label{eq:2.21}
u(x) =
\begin{cases}
\tilde{h}+\displaystyle\frac{h}{2}+\frac{h}{2}\sin{\left(\frac{\pi}{T}x+\frac{\pi}{2}\right)},
\quad\forall\; x\in [-T, 0),\\
\tilde{h}+h+\displaystyle\frac{1}{T}\tilde{f}''(0)x^2,
\quad\forall\; x\in [0, T],
\end{cases}
\end{equation}
where $$ h = 2\tilde{f}'(0) - T \tilde{f}''(0) \quad \text{and}\quad
\tilde{h} = \displaystyle\frac{\tilde{f}(0)}{T}
+\left(\frac{7}{12}T\tilde{f}''(0) - \frac{3}{2}
\tilde{f}'(0)\right).$$ It is easy to verify that the function $u$
defined by \eqref{eq:2.21} also satisfies all conditions in
\eqref{eq:2.8}.
\end{Remark}

\begin{Remark} \label{rem:2.2}
The solution of the TBVP \eqref{eq:2.4} (equivalently,
\eqref{eq:2.2}) is not unique. In fact, by the definition of $u(x)$
we observe that such a function $u$ is not unique. This results the
non-uniqueness of the initial velocity $v(x)$ and then the solution
$\tilde{y}=\tilde{y}(t,x)$. For example, choose a $C^1$-smooth
function $v_1(x)$ satisfying
\begin{equation} \label{eq:2.22}
\int_{x-T}^{x+T}{v_1(\tau)}d\tau = 0.
\end{equation}
This implies that $v_1(x)$ is a $2T$-periodic function and satisfies
\begin{equation} \label{eq:2.23}
\int_{-T}^{T}{v_1(\tau)}d\tau = 0.
\end{equation}
Obviously, the function
\begin{equation} \label{eq:2.24}
\tilde{y}(t,x) =
\displaystyle\frac{1}{2}\int_{x-t}^{x+t}\big(v(\tau)+v_1(\tau)\big)d\tau
\end{equation}
also gives a $C^2$-smooth solution of the TBVP \eqref{eq:2.4}.
\end{Remark}

The idea to construct $v(x)$ (see \eqref{eq:2.9}) essentially comes
from the {\it characteristic-quadrilateral identity} given in
\cite{kong3}. In fact, using the characteristic-quadrilateral
identity, we have
\begin{equation} \label{eq:2.25}
\tilde{y}(A) + \tilde{y}(D) = \tilde{y}(B_1) + \tilde{y}(C_1).
\end{equation}
See Figure 1. By the the initial condition in the TBVP
\eqref{eq:2.4}, we get
\begin{equation}\label{eq:2.26}
\tilde{y}(A) = \tilde{y}(B_1) + \tilde{y}(C_1).
\end{equation}
Similarly, we obtain
\begin{equation}\label{eq:2.27}
\tilde{y}(B_1) =  \tilde{y}(B_2) + \tilde{y}(C_2),\quad \cdots,\quad
\tilde{y}(B_{N-1}) =  \tilde{y}(B_N) + \tilde{y}(C_N).
\end{equation}
On the other hand, it follows from \eqref{eq:2.6} that
\begin{equation}\label{eq:2.28}
\tilde{y}(A)  =
\displaystyle\frac{1}{2}\int_{-T}^{x}{v(\tau)}d\tau,\quad
\tilde{y}(B_N) =
\displaystyle\frac{1}{2}\int_{-T}^{x-2NT}{v(\tau)}d\tau.
\end{equation}
Combining \eqref{eq:2.26}-\eqref{eq:2.28} gives the definition of
$v(x)$ shown by \eqref{eq:2.9}.

\begin{center}
\begin{figure}[htb]
\begin{minipage}[t]{1\linewidth}
\vspace{0pt}
 \centering
\centerline{
 \psfig{figure=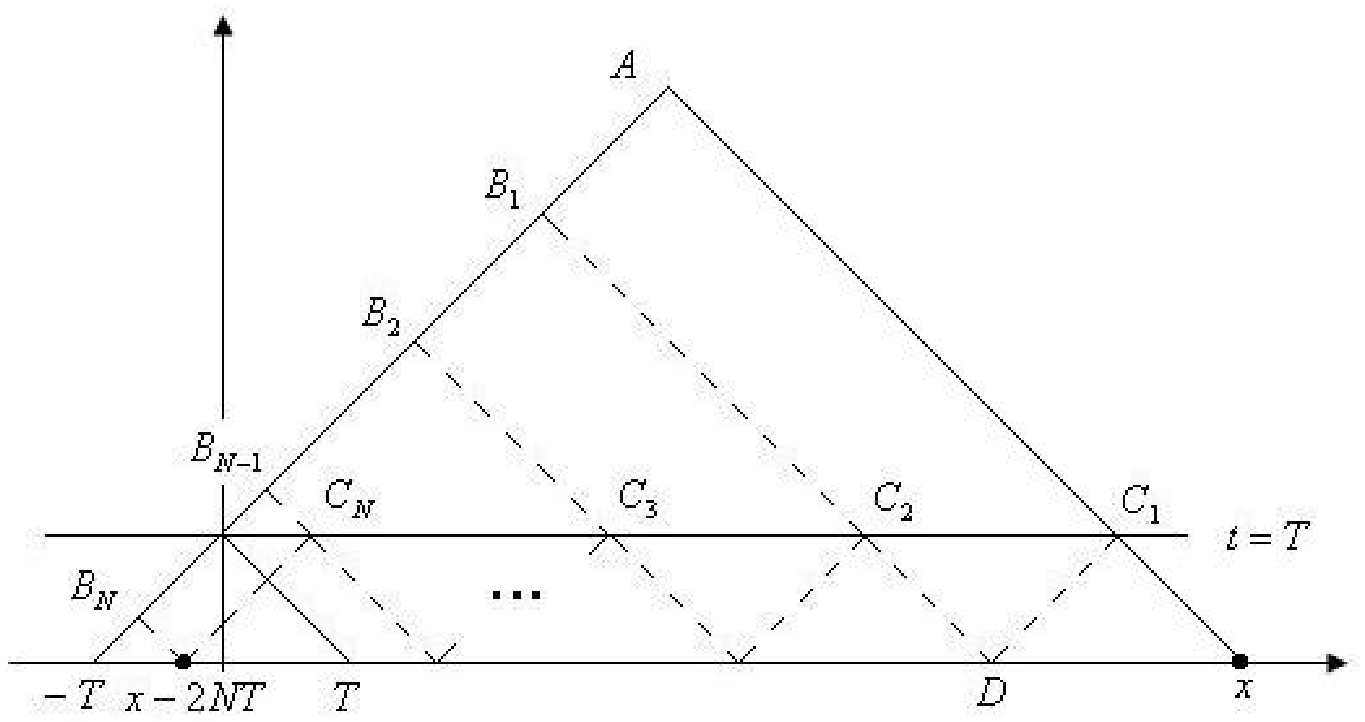,height=7cm}}
 \caption{The method to construct the initial velocity $v(x)$ by
characteristic-quadrilaterals}
\end{minipage}
\end{figure}
\end{center}

\begin{Remark} \label{rem:2.3}
In fact, the existence of the $C^2$-solution of the TBVP
\eqref{eq:2.4} is equivalent to the existence of the $C^1$-solution
of the following integral equation
\begin{equation}\label{eq:2.29}
\int_{x-T}^{x+T}{v(\tau)}d\tau = \tilde{f}(x).
\end{equation}
\eqref{eq:2.29} is a typical example of {\it Volterra integral
equations of the first kind}. By the theory of Volterra integral
equations, we can obtain the existence of the $C^1$-solution of the
integral equation \eqref{eq:2.29}. However, by this theory we do not
know how to construct the desired solution. But, in our theory we
present a direct method to construct the desired solution of the
TBVP \eqref{eq:2.1}.
\end{Remark}

\section{Wave equations in several space
variables} \label{sec:3}

In this section, we study the global exact controllability of wave
equations in several space variables and prove Theorem
\ref{thm:1-1.1}.

Consider the TBVP \eqref{eq:1-1.1}-\eqref{eq:1-1.3}, i.e., the
following TBVP
\begin{equation} \label{eq:3.1}
\begin{cases}
\square\, y(t,x) = 0,\\
y(0,x) = f(x),\\
y(T,x) = g(x),
\end{cases}
\end{equation}
where $(t,x)\in\mathds{R}^+\times\mathds{R}^n$, the symbol $\square$
stands for the d'Alembert's operator, i.e.,
$$\square =\partial_t^2-\sum_{i=1}^n\partial_i^2= \partial_t^2-\Delta_x,$$
and $f(x),g(x)\in C^{[n/2]+2}(\mathds{R}^{n})$ are two given
functions defined on $\mathds{R}^n$. In \eqref{eq:3.1}, without loss
of generality, we assume that the propagation speed $c$ of wave is
1. Thus, in order to prove Theorem \ref{thm:1-1.1}, it suffice to
show the following theorem.
\begin{Theorem}\label{thm:3.1}
The TBVP \eqref{eq:3.1} has a $C^2$-smooth solution $y = y(t,x)$
defined on the strip $[0,T]\times\mathds{R}^n$.
\end{Theorem}

\begin{Remark}\label{rem:3.1}
Theorem \ref{thm:3.1} implies the global exact controllability of
wave equations in several space variables. Similar result is true
for the following inhomogeneous wave equations $$ \square\, y(t,x) =
F(t,x).$$\end{Remark}

We next give the proof of Theorem \ref{thm:3.1} (equivalently,
Theorem \ref{thm:1-1.1}).

\begin{proof} We distinguish two cases to prove Theorem \ref{thm:3.1}.

{\bf Case A: $n=2k+1\; (k\in\mathds{N})$}

Define the spherical mean of any given function $h(x)$ defined on
$\mathds{R}^n$:
\begin{equation} \label{eq:3.2}
\mathcal{A}_r h(x) =
\displaystyle\frac{1}{\omega_{n-1}}\int_{S^{n-1}}{h(x+ry)}d\sigma(y),
\end{equation}
where $\omega_{n-1}$ denotes the area of the unit sphere
$S^{n-1}\subset\mathds{R}^n$ and $d\sigma(y)$ stands for the
Lebesgue measure on the unit sphere $S^{n-1}$. Let
\begin{equation} \label{eq:3.3}
w(t,r) = \left(\displaystyle\frac{1}{r}\frac{\partial}{\partial
r}\right)^{k-1}\left(r^{2k-1}\mathcal{A}_r y(t,x)\right).
\end{equation}
It is easy to check that the function $w(t,r)$ satisfies the
following TBVP
\begin{equation} \label{eq:3.4}
\begin{cases}
w_{tt}- w_{rr} = 0,\\
w(0,r) = \left(\displaystyle\frac{1}{r}\frac{\partial}{\partial
r}\right)^{k-1}\left(r^{2k-1}\mathcal{A}_r f(x)\right),\\
w(T,r) = \left(\displaystyle\frac{1}{r}\frac{\partial}{\partial
r}\right)^{k-1}\left(r^{2k-1}\mathcal{A}_r g(x)\right).
\end{cases}
\end{equation}
Thus, by Theorem \ref{thm:2.1}, the TBVP \eqref{eq:3.4} has a
$C^2$-smooth solution defined on the domain $[0,T]\times\mathds{R}$.

Denote
\begin{equation} \label{eq:3.5}
c_0 = \prod_{i=1}^k{(2i-1)}.
\end{equation}
Then we have
\begin{equation} \label{eq:3.6}
y(t,x) = \lim_{r\rightarrow 0}\mathcal{A}_r y(t,x) =
\lim_{r\rightarrow 0} \displaystyle\frac{1}{c_0 r}w(t,r).
\end{equation}
It is easy to verify that the function $y=y(t,x)$ defined by
\eqref{eq:3.6} is a $C^2$-smooth solution of the TBVP \eqref{eq:3.1}
on the strip $[0,T]\times\mathds{R}^n$.

{\bf Case B: $n=2k\; (k\in\mathds{N})$}

By Hadamard's method of descent, we can also prove that the TBVP
\eqref{eq:3.1} has a $C^2$-smooth solution on the strip
$[0,T]\times\mathds{R}^n$. The main idea here is that if $y$ solves
a wave equation with $n$ space variables, then it is also a solution
of the corresponding wave equation with $n+1$ space variables, which
happens to be independent of the last variable $x_{n+1}$. Here we
omit the details.  Thus the proof of Theorem \ref{thm:3.1} is
completed.
\end{proof}

\begin{Remark} \label{rem:3.2}
Noting Remark \ref{rem:2.2}, we know that the solution of the TBVP
\eqref{eq:3.1} does not possesses the uniqueness.\end{Remark}

\section{Wave equation defined on a closed curve}\label{sec:4}

This section concerns the exact controllability of the wave equation
defined on a circle. In other words, in this section we prove
Theorem \ref{thm:1-2.1} and then Theorem \ref{thm:1-2.2}.

To do so, we need the following Lemma (see Kong \cite{kong}).

\begin{Lemma} \label{lem:4.1}
Suppose that $F(x)$ is a $L$-periodic $C^2$ function of
$x\in\mathds{R}$, and its derivative of third order, i.e.,
$F'''(x)$, is piecewise smooth. Suppose furthermore that the Fourier
series of $F(x)$ is given by
\begin{equation} \label{eq:4.1}
F(x) = \displaystyle\frac{1}{2}A_0 + \sum_{k=1}^\infty{A_k
\cos{\left(\frac{2k\pi}{L}x\right)}} + \sum_{k=1}^\infty{B_k
\sin{\left(\frac{2k\pi}{L}x\right)}},
\end{equation}
where
$$A_0=\displaystyle\frac{2}{L}\int_0^L{f(x)}dx,\quad
A_k=\displaystyle\frac{2}{L}\int_0^L{f(x)\cos{\left(\frac{2k\pi}{L}x\right)}}dx,\quad
B_k=\frac{2}{L}\int_0^L{f(x)\sin{\left(\frac{2k\pi}{L}x\right)}}dx.$$
Then the series $\displaystyle\sum_{k=1}^\infty{|k^2 A_k|}$ and
$\displaystyle\sum_{k=1}^\infty{|k^2 B_k|}$ are convergent.
\end{Lemma}

\begin{proof} It follows from \eqref{eq:4.1} that
\begin{equation} \label{eq:4.2}
\begin{cases}
\displaystyle\sum_{k=1}^\infty{B_k
\sin{\left(\frac{2k\pi}{L}x\right)}}=\frac{1}{2}\big(F(x)-F(-x)\big)\triangleq G(x),\\
\displaystyle\frac{1}{2}A_0 + \sum_{k=1}^\infty{A_k
\cos{\left(\frac{2k\pi}{L}x\right)}}=\frac{1}{2}\big(F(x)+F(-x)\big)\triangleq
H(x).
\end{cases}
\end{equation}
Since $G(x)$ is a $L$-periodic odd function, its derivative of third
order, i.e., $G'''(x)$, is a $L$-periodic piecewise continuous even
function. So the form of the Fourier series of $G'''(x)$ should be
\begin{equation} \label{eq:4.3}
G'''(x) = \displaystyle\frac{1}{2}B_0^{(3)} +
\sum_{k=1}^\infty{B_k^{(3)} \cos{\left(\frac{2k\pi}{L}x\right)}},
\end{equation}
where $B_0^{(3)}$ and $B_k^{(3)}\;(k=1,2,\cdots)$ stand for the
coefficients of the Fourier series. By Parseval inequality, we
obtain
\begin{equation} \label{eq:4.4}
\displaystyle\frac{1}{2}\left(B_0^{(3)}\right)^2 +
\sum_{k=1}^\infty{\left(B_k^{(3)}\right)^2}=\frac{2}{L}\int_0^L{\left[G'''(x)\right]^2}dx<\infty.
\end{equation}
On the other hand,
\begin{align} \label{eq:4.5}
B_k^{(3)} &= \frac{2}{L}\int_0^L{G'''(x)\cos{\left(\frac{2k\pi}{L}x\right)}dx} \nonumber \\
&=
\frac{2}{L}\left[G''(x)\cos{\left(\frac{2k\pi}{L}x\right)}\right]\Big|_0^L
+
\frac{2k\pi}{L}\frac{2}{L}\int_0^L{G''(x)\sin{\left(\frac{2k\pi}{L}x\right)}dx} \nonumber \\
&=
\frac{2k\pi}{L}\frac{2}{L}\left[G'(x)\sin{\left(\frac{2k\pi}{L}x\right)}\right]\Big|_0^L
-
\left(\frac{2k\pi}{L}\right)^2\frac{2}{L}\int_0^L{G'(x)\cos{\left(\frac{2k\pi}{L}x\right)}dx} \nonumber \\
&=
-\left(\frac{2k\pi}{L}\right)^2\frac{2}{L}\left[G(x)\cos{\left(\frac{2k\pi}{L}x\right)}\right]\Big|_0^L
-
\left(\frac{2k\pi}{L}\right)^3\frac{2}{L}\int_0^L{G(x)\sin{\left(\frac{2k\pi}{L}x\right)}dx} \nonumber \\
&=-\left(\frac{2k\pi}{L}\right)^3B_k.
\end{align}
Here we have made use of the fact that $G(x)$ and $G''(x)$ are
$L$-periodic odd functions. So
\begin{equation} \label{eq:4.6}
\displaystyle\sum_{k=1}^\infty{\left(k^3 B_k\right)^2}=
\left(\frac{L}{2\pi}\right)^6\sum_{k=1}^\infty{\left(B_k^{(3)}\right)^2}
<\infty.
\end{equation}
Then by Cauchy inequality, it yields
\begin{equation} \label{eq:4.7}
\displaystyle\sum_{k=1}^\infty{\left|k^2 B_k\right|} \leq
\sqrt{\sum_{k=1}^\infty{\left(k^3 B_k\right)^2}\cdot
\sum_{k=1}^\infty{\frac{1}{k^2}}}<\infty.
\end{equation}

Similarly, we can show that
$\displaystyle\sum_{k=1}^\infty{\left|k^2 A_k\right|}$ is
convergent.

Thus, the proof of Lemma \ref{lem:4.1} is completed.
\end{proof}

\begin{Remark} \label{rem:4.1}
Obviously, it follows from Lemma \ref{lem:4.1} that the series
$$\displaystyle\sum_{k=1}^\infty{|A_k|},\quad
\displaystyle\sum_{k=1}^\infty{|B_k|},\quad
\displaystyle\sum_{k=1}^\infty{|k A_k|},\quad
\displaystyle\sum_{k=1}^\infty{|k B_k|}$$ are convergent.
\end{Remark}

We now prove Theorem \ref{thm:1-2.1}.

\begin{proof}
Suppose that the Fourier series of the solution $y=y(t,\theta)$ to
\eqref{eq:1-2.5} reads
\begin{equation} \label{eq:4.8}
y(t,\theta) = \displaystyle\frac{1}{2}a_0(t) +
\sum_{k=1}^\infty{a_k(t) \cos{\left(\frac{2k\pi}{L}\theta\right)}} +
\sum_{k=1}^\infty{b_k(t) \sin{\left(\frac{2k\pi}{L}\theta\right)}},
\end{equation}
where $a_0(t),\;a_k(t),\;b_k(t)$ stand for the coefficients of the
Fourier series. Then by the first equation in \eqref{eq:1-2.5}, we
obtain
\begin{equation} \label{eq:4.9}
\begin{cases}
a_0''(t)=0,\\
a_k''(t)+\displaystyle\left(\frac{2k\pi}{L}\right)^2a_k(t)=0,\\
b_k''(t)+\displaystyle\left(\frac{2k\pi}{L}\right)^2b_k(t)=0.
\end{cases}
\end{equation}
In particular,
\begin{equation} \label{eq:4.10}
\begin{cases}
f(\theta) = \displaystyle\frac{1}{2}a_0(0) +
\sum_{k=1}^\infty{a_k(0) \cos{\left(\frac{2k\pi}{L}\theta\right)}} +
\sum_{k=1}^\infty{b_k(0) \sin{\left(\frac{2k\pi}{L}\theta\right)}},\\
g(\theta) = \displaystyle\frac{1}{2}a_0(T) +
\sum_{k=1}^\infty{a_k(T) \cos{\left(\frac{2k\pi}{L}\theta\right)}} +
\sum_{k=1}^\infty{b_k(T) \sin{\left(\frac{2k\pi}{L}\theta\right)}},
\end{cases}
\end{equation}
in which
\begin{equation} \label{eq:4.11}
\begin{cases}
a_0(0)=\displaystyle\frac{2}{L}\int_0^L{f(x)}dx,\\
a_0(T)=\displaystyle\frac{2}{L}\int_0^L{g(x)}dx,
\end{cases}
\end{equation}
\begin{equation} \label{eq:4.12}
\begin{cases}
a_k(0)=\displaystyle\frac{2}{L}\int_0^L{f(x)\cos{\left(\frac{2k\pi}{L}x\right)}}dx,\\
a_k(T)=\displaystyle\frac{2}{L}\int_0^L{g(x)\cos{\left(\frac{2k\pi}{L}x\right)}}dx,
\end{cases}
\end{equation}
and
\begin{equation} \label{eq:4.13}
\begin{cases}
b_k(0)=\displaystyle\frac{2}{L}\int_0^L{f(x)\sin{\left(\frac{2k\pi}{L}x\right)}}dx,\\
b_k(T)=\displaystyle\frac{2}{L}\int_0^L{g(x)\sin{\left(\frac{2k\pi}{L}x\right)}}dx.
\end{cases}
\end{equation}
It follows from \eqref{eq:4.9} that
\begin{equation} \label{eq:4.14}
\begin{cases}
a_0(t)=\alpha_0 + \beta_0 t,\\
a_k(t)=\displaystyle\alpha_k\cos{\left(\frac{2k\pi}{L}t\right)} +
\beta_k\sin{\left(\frac{2k\pi}{L}t\right)},\\
b_k(t)=\displaystyle\bar{\alpha}_k\cos{\left(\frac{2k\pi}{L}t\right)}
+ \bar{\beta}_k\sin{\left(\frac{2k\pi}{L}t\right)},
\end{cases}
\end{equation}
where $\alpha_0,\;\beta_0,\; \alpha_k,\;\beta_k,\;\bar{\alpha}_k$
and $\bar{\beta}_k$ are some constants. Comparing
\eqref{eq:4.11}-\eqref{eq:4.13} with \eqref{eq:4.14} gives
\begin{equation} \label{eq:4.15}
\begin{cases}
\alpha_0=\displaystyle\frac{2}{L}\int_0^L{f(x)}dx,\\
\alpha_0 + \beta_0 T=\displaystyle\frac{2}{L}\int_0^L{g(x)}dx,
\end{cases}
\end{equation}
\begin{equation} \label{eq:4.16}
\begin{cases}
\alpha_k=\displaystyle\frac{2}{L}\int_0^L{f(x)\cos{\left(\frac{2k\pi}{L}x\right)}}dx,\\
\displaystyle\alpha_k\cos{\left(\frac{2k\pi}{L}T\right)} +
\beta_k\sin{\left(\frac{2k\pi}{L}T\right)}=\displaystyle\frac{2}{L}\int_0^L{g(x)\cos{\left(\frac{2k\pi}{L}x\right)}}dx
\end{cases}
\end{equation}
and
\begin{equation} \label{eq:4.17}
\begin{cases}
\bar{\alpha}_k=\displaystyle\frac{2}{L}\int_0^L{f(x)\sin{\left(\frac{2k\pi}{L}x\right)}}dx,\\
\displaystyle\bar{\alpha}_k\cos{\left(\frac{2k\pi}{L}T\right)} +
\bar{\beta}_k\sin{\left(\frac{2k\pi}{L}T\right)}=\displaystyle\frac{2}{L}\int_0^L{g(x)\sin{\left(\frac{2k\pi}{L}x\right)}}dx.
\end{cases}
\end{equation}
That is,
\begin{equation} \label{eq:4.18}
\begin{cases}
\alpha_0=\displaystyle\frac{2}{L}\int_0^L{f(x)}dx,\\
\beta_0
=\displaystyle\frac{2}{LT}\int_0^L{\left[g(x)-f(x)\right]}dx,
\end{cases}
\end{equation}
\begin{equation} \label{eq:4.19}
\begin{cases}
\alpha_k=\displaystyle\frac{2}{L}\int_0^L{f(x)\cos{\left(\frac{2k\pi}{L}x\right)}}dx,\\
\beta_k=
\begin{cases}
0,\quad {\rm if}\;\; \displaystyle\frac{2kT}{L}\in \mathds{N},\\
\displaystyle\frac{2}{L\sin{\left(\frac{2k\pi}{L}T\right)}}
\int_0^L{\left[g(x)-f(x)\cos{\left(\frac{2k\pi}{L}T\right)}\right]\cos{\left(\frac{2k\pi}{L}x\right)}}dx,\quad
{\rm if}\;\; \displaystyle\frac{2kT}{L}\not\in \mathds{N}
\end{cases}
\end{cases}
\end{equation}
and
\begin{equation} \label{eq:4.20}
\begin{cases}
\bar\alpha_k=\displaystyle\frac{2}{L}\int_0^L{f(x)\sin{\left(\frac{2k\pi}{L}x\right)}}dx,\\
\bar\beta_k=
\begin{cases}
0,\quad {\rm if}\;\;  \displaystyle\frac{2kT}{L}\in \mathds{N},\\
\displaystyle\frac{2}{L\sin{\left(\frac{2k\pi}{L}T\right)}}
\int_0^L{\left[g(x)-f(x)\cos{\left(\frac{2k\pi}{L}T\right)}\right]\sin{\left(\frac{2k\pi}{L}x\right)}}dx,\quad
{\rm if}\;\; \displaystyle\frac{2kT}{L}\not\in \mathds{N}.
\end{cases}
\end{cases}
\end{equation}
Here we would like to point out that, when
$\displaystyle\frac{2kT}{L}\in \mathds{N}$, i.e.,
$\displaystyle\sin{\left(\frac{2k\pi}{L}T\right)}=0$, in
\eqref{eq:4.19}-\eqref{eq:4.20} we take the corresponding
coefficients $\beta_k$ and $\bar\beta_k$ as zero. In fact, we may
also take $\beta_k$ and $\bar\beta_k$ as some series satisfying
\begin{equation} \label{eq:4.21}
\displaystyle\sum_{k=1}^\infty{|k^2\beta_k|}<\infty,\quad
\sum_{k=1}^\infty{|k^2 \bar\beta_k|}<\infty.
\end{equation} The
different choice of $\beta_k$ and $\bar\beta_k$ can give the
different solution. This implies that the solution of the TBVP under
consideration is not unique.

Since $\displaystyle\frac{T}{L}$ is a rational number and
$\displaystyle\frac{2T}{L}\not\in \mathds{N}$,
$\displaystyle\frac{2T}{L}$ can be expressed as a fraction
$\displaystyle\frac{p}{q}$, where $p$ and $q$ are two irreducible
integers and $q$ is not less than $2$, i.e., $q\geq 2$. By the
property of sinusoid, we have
\begin{equation} \label{eq:4.22}
\left|\displaystyle\sin{\left(\frac{2k\pi}{L}T\right)}\right|=
\left|\displaystyle\sin{\left(\frac{kp\pi}{q}\right)}\right|\geq
\left|\displaystyle\sin{\left(\frac{\pi}{q}\right)}\right|\triangleq
C_s,\quad \text{if $k$ is not the multiple of $q$}.
\end{equation}
So for the given $T$ and $L$, $C_s$ is a constant. Then it follows
from \eqref{eq:4.14}, \eqref{eq:4.19}, \eqref{eq:4.21} and Remark
\ref{rem:4.1} that
\begin{align} \label{eq:4.23}
\sum_{k=1}^\infty{\left|a_k(t)\right|}
&\leq\sum_{k=1}^\infty{\left(\left|\alpha_k\right|+\left|\beta_k\right|\right)} \nonumber \\
&\leq\displaystyle\left(1+\frac{1}{C_s}\right)
\sum_{k=1}^\infty{\left|\frac{2}{L}\int_0^L{f(x)\cos{\left(\frac{2k\pi}{L}x\right)}}dx\right|}
+\frac{1}{C_s}\sum_{k=1}^\infty{\left|\frac{2}{L}\int_0^L{g(x)\cos{\left(\frac{2k\pi}{L}x\right)}}dx\right|} \nonumber \\
&<\infty.
\end{align}

Similarly, we obtain from \eqref{eq:4.14}, \eqref{eq:4.20},
\eqref{eq:4.21} and Remark \ref{rem:4.1} that
\begin{equation} \label{eq:4.24}
\sum_{k=1}^\infty{\left|b_k(t)\right|} <\infty.
\end{equation}
Combining \eqref{eq:4.23}, \eqref{eq:4.24} and \eqref{eq:4.15} gives
the convergence of the Fourier series \eqref{eq:4.8} of the solution
$y=y(t,\theta)$ immediately.

Moreover, by Lemma \ref{lem:4.1}, Remark \ref{rem:4.1} and
\eqref{eq:4.21} we obtain
\begin{equation} \label{eq:4.25}
\sum_{k=1}^\infty{\left|k \alpha_k\right|} <\infty,\quad
\sum_{k=1}^\infty{\left|k \beta_k\right|} <\infty,\quad
\sum_{k=1}^\infty{\left|k \bar\alpha_k\right|} <\infty,\quad
\sum_{k=1}^\infty{\left|k \bar\beta_k\right|} <\infty
\end{equation}
and
\begin{equation} \label{eq:4.26}
\sum_{k=1}^\infty{\left|k^2 \alpha_k\right|} <\infty,\quad
\sum_{k=1}^\infty{\left|k^2 \beta_k\right|} <\infty,\quad
\sum_{k=1}^\infty{\left|k^2 \bar\alpha_k\right|} <\infty,\quad
\sum_{k=1}^\infty{\left|k^2 \bar\beta_k\right|} <\infty.
\end{equation}
Using \eqref{eq:4.25} and \eqref{eq:4.26}, by a similar argument as
used above, we can prove
\begin{equation} \label{eq:4.27}
\sum_{k=1}^\infty{\left|k a_k(t)\right|} <\infty,\quad
\sum_{k=1}^\infty{\left|k b_k(t)\right|} <\infty,\quad
\sum_{k=1}^\infty{\left|a_k'(t)\right|} <\infty,\quad
\sum_{k=1}^\infty{\left|b_k'(t)\right|} <\infty
\end{equation}
and
\begin{equation}\left\{\begin{array}{l} \label{eq:4.28}
{\displaystyle \sum_{k=1}^\infty{\left|k^2 a_k(t)\right|}
<\infty,\quad \sum_{k=1}^\infty{\left|k^2 b_k(t)\right|}
<\infty,\quad \sum_{k=1}^\infty{\left|k a_k'(t)\right|}
<\infty,\quad
\sum_{k=1}^\infty{\left|k b_k'(t)\right|} <\infty, }\vspace{2mm}\\
{\displaystyle \sum_{k=1}^\infty{\left|a_k''(t)\right|}
<\infty,\quad \sum_{k=1}^\infty{\left|b_k''(t)\right|} <\infty.}
\end{array}\right.\end{equation}
Obviously, \eqref{eq:4.27} and \eqref{eq:4.28} imply that the
Fourier series of the first-order derivatives of $y$ (i.e., $y_t$
and $y_{\theta}$) and second-order derivatives of $y$ (i.e.,
$y_{tt},\;y_{t\theta}$ and $y_{\theta\theta}$) are also convergent,
respectively.

Thus, the proof of Theorem \ref{thm:1-2.1} is completed.
\end{proof}

\begin{Remark} \label{rem:4.2}
By \eqref{eq:4.10}, \eqref{eq:4.14}, \eqref{eq:4.18}, the first
equation in \eqref{eq:4.19} and the first equation in
\eqref{eq:4.20}, if $\displaystyle\frac{2T}{L}\in \mathds{N}$, then
$g(\theta)$ satisfies
\begin{align} \label{eq:4.29}
g(\theta) &= \displaystyle\frac{1}{2}\left(\alpha_0+\beta_0 T\right)
+ \sum_{k=1}^\infty{\alpha_k \cos{\left(\frac{2k\pi T}{L}\right)}
\cos{\left(\frac{2k\pi}{L}\theta\right)}} +
\sum_{k=1}^\infty{\bar{\alpha}_k \cos{\left(\frac{2k\pi
T}{L}\right)} \sin{\left(\frac{2k\pi}{L}\theta\right)}} \nonumber \\
&=\displaystyle\frac{1}{L}\int_0^L{g(x)}dx +
\sum_{k=1}^\infty{\frac{2}{L}
\int_0^L{f(x)\cos{\left(\frac{2k\pi}{L}x\right)}}dx\cos{\left(\frac{2k\pi}{L}\theta\right)}\cos{\left(\frac{2k\pi
T}{L}\right)}} + \nonumber \\
&\qquad\qquad\qquad\quad\sum_{k=1}^\infty{\frac{2}{L}
\int_0^L{f(x)\sin{\left(\frac{2k\pi}{L}x\right)}}dx\sin{\left(\frac{2k\pi}{L}\theta\right)}
\cos{\left(\frac{2k\pi T}{L}\right)}}.
\end{align}
Then in this case, for any given $f(\theta)$, the terminal value
$g(\theta)$ can not be arbitrary. In particular, if
$\displaystyle\frac{T}{L}\in \mathds{N}$, then it should hold that
\begin{equation} \label{eq:4.30}
g(\theta) =
\displaystyle\frac{1}{L}\int_0^L{\left[g(x)-f(x)\right]}dx +
f(\theta).
\end{equation}
\end{Remark}

\begin{Remark} \label{rem:4.3}
If $\displaystyle\frac{T}{L}$ is a irrational number, then
\eqref{eq:4.22} is incorrect, and then we can not get the
convergence in \eqref{eq:4.23}, etc.
\end{Remark}

\begin{Remark} \label{rem:4.4}
Theorem \ref{thm:1-2.1} shows that, for some $T$ satisfied the
assumptions in Theorem \ref{thm:1-2.1}, there exist some initial
velocity $v(\theta)$ such that the Cauchy problem for the wave
equation in \eqref{eq:1-2.5} with the initial data
\begin{equation} \label{eq:4.31}
t=0:\;\;y=f(\theta),\quad y_t=v(\theta)
\end{equation}
has a solution $y=y(t,\theta)\in C^2([0,T]\times\mathds{R})$ which
satisfies the terminal condition in \eqref{eq:1-2.5}, i.e., the
third equation in \eqref{eq:1-2.5}.
\end{Remark}

Theorem \ref{thm:1-2.2} comes from Theorem \ref{thm:1-2.1} directly.

\section{Wave equation with homogeneous Dirichlet or Neumann boundary
conditions} \label{sec:5}

In this section, we investigate the TBVP for the wave equation
defined on the domain $[0,T]\times[0,L]$ with homogeneous Dirichlet
boundary conditions and homogeneous Neumann boundary conditions,
respectively, where $T,\;L>0$ are two given real numbers.

More precisely, we consider the following problem for the wave
equation with the homogeneous Dirichlet boundary conditions
\begin{equation} \label{eq:5.1}
\begin{cases}
y_{tt}-y_{xx} = 0,\\
y(0,x) = f(x),\\
y(T,x) = g(x),\\
y(t,0) = y(t,L) = 0,
\end{cases}
\end{equation}
and the problem for the wave equation with the homogeneous Neumann
boundary conditions
\begin{equation} \label{eq:5.2}
\begin{cases}
y_{tt}-y_{xx} = 0,\\
y(0,x) = f(x),\\
y(T,x) = g(x),\\
y_x(t,0) = y_x(t,L) = 0,
\end{cases}
\end{equation}
where $y=y(t,x)$ is the unknown function of
$(t,x)\in[0,T]\times[0,L]$, $f(x)$ and $g(x)$ are two given $C^{3}$
functions of $x\in [0,L]$. Moreover, $f$ and $g$ satisfy the
compatibility conditions
\begin{equation} \label{eq:5.3}
f(x)\Big|^{x=0}_{x=L}=g(x)\Big|^{x=0}_{x=L}=0,\quad
f''(x)\Big|^{x=0}_{x=L}=g''(x)\Big|^{x=0}_{x=L}=0, \quad\text{for
the Dirichlet conditions};
\end{equation}

\begin{equation} \label{eq:5.4}
f'(x)\Big|^{x=0}_{x=L}=g'(x)\Big|^{x=0}_{x=L}=0,\quad\text{for the
Neumann conditions}.
\end{equation}
We have the following theorem which can be viewed as consequences of
Theorem \ref{thm:1-2.1}.

\begin{Theorem}\label{thm:5.1}
(A) Suppose that the compatibility conditions in \eqref{eq:5.3} are
satisfied, $\displaystyle\frac{T}{L}$ is a rational number and
$\displaystyle\frac{T}{L}\not\in \mathds{N}$. Then the problem
\eqref{eq:5.1} admits a $C^2$ solution $y=y(t,x)$ on the domain
$[0,T]\times [0,L]$.

(B) Suppose that the compatibility conditions in \eqref{eq:5.4} are
satisfied, $\displaystyle\frac{T}{L}$ is a rational number and
$\displaystyle\frac{T}{L}\not\in \mathds{N}$. Then the problem
\eqref{eq:5.2} admits a $C^2$ solution $y=y(t,x)$ on the domain
$[0,T]\times [0,L]$.
\end{Theorem}

\begin{proof}
For the case of the Dirichlet boundary conditions, we extend any
$C^2$ solution $y=y(t,x)$ to the domain $[0,T]\times [-L,L]$ by
\begin{equation} \label{eq:5.5}
y(t,x)=-y(t,-x), \quad \text{for $x\in[-L,0]$},
\end{equation}
and then extend $y=y(t,x)$ to be $2L$-periodic. See \cite{kong3} and
\cite{kong4}. One easily checks that if the given initial/terminal
data have the form in \eqref{eq:5.1}, the extended initial/terminal
data are $2L$-periodic and given by (see \cite{kong3}-\cite{kong4})
\begin{equation} \label{eq:5.6}
\tilde{f}(x)\triangleq y(0,x)=
\begin{cases}
-f(-x),\quad \text{for $x\in[-L,0]$,}\\
f(x),\quad \text{for $x\in[0,L]$,}
\end{cases}
\end{equation}

\begin{equation} \label{eq:5.7}
\tilde{g}(x)\triangleq y(T,x)=
\begin{cases}
-g(-x),\quad \text{for $x\in[-L,0]$,}\\
g(x),\quad \text{for $x\in[0,L]$,}
\end{cases}
\end{equation}
respectively. Thus, we obtain an extended TBVP for the wave equation
defined on the strip $[0,T]\times \mathds{R}$. When the
compatibility conditions in \eqref{eq:5.3} are satisfied, this
extended $y=y(t,x)$ is a $C^2$ solution of the extended TBVP with
the extended initial/terminal data
$\left(\tilde{f}(x),\tilde{g}(x)\right)$. Therefore, we may make use
of Theorem \ref{thm:1-2.1} and obtain the solution, denoted by
$\tilde{y}=\tilde{y}(t,x)$, to the extended TBVP defined on the
strip $[0,T]\times \mathds{R}$. Obviously,
$\tilde{y}=\tilde{y}(t,x)$ is a $2L$-periodic odd $C^2$ function. So
the Dirichlet boundary conditions are satisfied naturally. Let
$y=y(t,x)$ be the restriction of $\tilde{y}=\tilde{y}(t,x)$ on the
region $[0,T]\times[0,L]$. It is easy to see that $y=y(t,x)$ is the
$C^2$ solution to \eqref{eq:5.1}.

Similarly, for the case of Neumann boundary conditions, we can
extend $y=y(t,x)$ by
\begin{equation} \label{eq:5.8}
y(t,x)=y(t,-x), \quad \text{for $x\in[-L,0]$}.
\end{equation}
See \cite{kong3}-\cite{kong4}. Then, $y=y(t,x)$ can be extended to
be a classical $2L$-periodic $C^2$ solution of the extended TBVP for
the wave equation with $2L$-periodic initial/terminal data given by
\begin{equation} \label{eq:5.9}
\tilde{f}(x)\triangleq y(0,x)=
\begin{cases}
f(-x),\quad \text{for $x\in[-L,0]$,}\\
f(x),\quad \text{for $x\in[0,L]$,}
\end{cases}
\end{equation}

\begin{equation} \label{eq:5.10}
\tilde{g}(x)\triangleq y(T,x)=
\begin{cases}
g(-x),\quad \text{for $x\in[-L,0]$,}\\
g(x),\quad \text{for $x\in[0,L]$,}
\end{cases}
\end{equation}
respectively. When the compatibility conditions in \eqref{eq:5.4}
are satisfied, by a similar argument as used above, we can prove the
part (B) in Theorem \ref{thm:5.1}.

Thus, the proof of Theorem \ref{thm:5.1} is completed.
\end{proof}

\begin{Remark} \label{rem:5.1}
Theorem \ref{thm:5.1} shows that the wave equation defined on the
domain $[0,T]\times [0,L]$ still possesses the exact controllability
in the cases of homogeneous Dirichlet boundary conditions and
homogeneous Neumann boundary conditions, provided that the
corresponding compatibility conditions are satisfied.\end{Remark}

\section{Wave equation with inhomogeneous
Dirichlet or Neumann boundary conditions} \label{sec:6}

This section is devoted to the exact controllability of the wave
equation defined on the domain $[0,T]\times[0,L]$ with inhomogeneous
Dirichlet boundary conditions and inhomogeneous Neumann boundary
conditions, respectively, where $T,\;L$ are two given positive real
numbers.

More precisely, we consider the following TBVP for the wave equation
with the inhomogeneous Dirichlet boundary conditions
\begin{equation} \label{eq:6.1}
\begin{cases}
y_{tt}-y_{xx} = 0,\\
y(0,x) = f(x),\\
y(T,x) = g(x),\\
y(t,0) = h(t),\\
y(t,L) = l(t),
\end{cases}
\end{equation}
and the TBVP for the wave equation with the inhomogeneous Neumann
boundary conditions
\begin{equation} \label{eq:6.2}
\begin{cases}
y_{tt}-y_{xx} = 0,\\
y(0,x) = f(x),\\
y(T,x) = g(x),\\
y_x(t,0) = H(t),\\
y_x(t,L) = K(t),
\end{cases}
\end{equation}
where $y=y(t,x)$ is the unknown function of $(t,x)\in
[0,T]\times[0,L]$, $f(x), g(x)$ are two given $C^{3}$ functions of
$x\in [0,L]$ which stands for the initial data and terminal data,
respectively, $h(t),\; l(t)$ are two given $C^{3}$ functions of
$t\in [0,T]$, and $H(t),\; K(t)$ are two given $C^{2}$ functions of
$t\in [0,T]$. Moreover, we assume that $f$, $g$, $h$, $l$, $H$ and
$K$ satisfy the following compatibility conditions
\begin{equation} \label{eq:6.3}
\begin{cases}
f(0) = h(0),\quad f''(0) = h''(0),\\
f(L) = l(0),\quad f''(L) = l''(0),\\
g(0) = h(T),\quad g''(0) = h''(T),\\
g(L) = l(T),\quad g''(L) = l''(T),
\end{cases}
\quad\text{for the Dirichlet conditions};
\end{equation}
\begin{equation} \label{eq:6.4}
\begin{cases}
f'(0) = H(0),\\
f'(L) = K(0),\\
g'(0) = H(T),\\
g'(L) = K(T),
\end{cases}
\quad\text{for the Neumann conditions}.
\end{equation}
We have
\begin{Theorem} \label{thm:6.1}
(A) Suppose that the compatibility conditions in \eqref{eq:6.3} are
satisfied, $\displaystyle\frac{T}{L}$ is a rational number and
$\displaystyle\frac{T}{L}\not\in\mathds{N}$. Then the problem
\eqref{eq:6.1} admits a $C^2$ solution $y=y(t,x)$ on the domain
$[0,T]\times[0,L]$.

(B) Suppose that the compatibility conditions in \eqref{eq:6.4} are
satisfied, $\displaystyle\frac{T}{L}$ is a rational number and $T<
L$. Then the problem \eqref{eq:6.2} admits a $C^2$ solution
$y=y(t,x)$ on the domain $[0,T]\times[0,L]$.
\end{Theorem}

\begin{proof}
For the problem \eqref{eq:6.1}, noting that
$\displaystyle\frac{T}{L}\not\in\mathds{N}$, without loss of
generality, we may assume that $T<L$ (otherwise, we only need
exchange the $t$-axes and the $x$-axes and then reduce the problem
to the case of $T<L$). In this situation, we can extend $h(t)$ to
the interval $[-T-L,T+L]$ by
\begin{equation} \label{eq:6.5}
h(t+L)+h(t-L)=2l(t), \quad \text{for $t\in[0,T]$}.
\end{equation}
Moreover we require that the extended $h(t)$ is a $C^3$ function on
the interval $[-T-L,T+L]$.

Introduce
\begin{equation} \label{eq:6.6}
\tilde{y}(t,x)=y(t,x)-\displaystyle\frac{h(t+x)+h(t-x)}{2}.
\end{equation}
Then by \eqref{eq:6.5}, the problem \eqref{eq:6.1} becomes
\begin{equation} \label{eq:6.7}
\begin{cases}
\tilde{y}_{tt}-\tilde{y}_{xx} = 0,\\
\tilde{f}(x) \triangleq\tilde{y}(0,x) = f(x)-\displaystyle\frac{h(x)+h(-x)}{2},\\
\tilde{g}(x) \triangleq\tilde{y}(T,x) = g(x)-\displaystyle\frac{h(T+x)+h(T-x)}{2},\\
\tilde{h}(t) \triangleq\tilde{y}(t,0) = 0,\\
\tilde{l}(t) \triangleq\tilde{y}(t,L) = 0.
\end{cases}
\end{equation}
And by \eqref{eq:6.3}, \eqref{eq:6.5}, the boundary data
$\tilde{f}$, $\tilde{g}$, $\tilde{h}$ and $\tilde{l}$ satisfy the
compatibility conditions
\begin{equation} \label{eq:6.8}
\begin{cases}
\tilde{f}(0) = \tilde{f}(L) = 0,\quad \tilde{f}''(0) = \tilde{f}''(L) = 0,\\
\tilde{g}(0) = \tilde{g}(L) = 0,\quad \tilde{g}''(0) =
\tilde{g}''(L) = 0.
\end{cases}
\end{equation}
Therefore, we can make use of Theorem \ref{thm:5.1} (A) and obtain
the $C^2$ solution to the problem \eqref{eq:6.7}, and then the $C^2$
solution to \eqref{eq:6.1}. This proves the part (A) in Theorem
\ref{thm:6.1}.

Similarly, for the problem \eqref{eq:6.2}, noting $T<L$, we can
extend $H(t)$ to the interval $[-L,T+L]$ by
\begin{equation} \label{eq:6.9}
H(t+L)+H(t-L)=2K(t), \quad \text{for $t\in[0,T]$}.
\end{equation}
As before, we require that the extended $H(t)$ is a $C^2$ function
on $[-L,T+L]$.

Let
\begin{equation} \label{eq:6.10}
\tilde{y}(t,x)=y(t,x)-\displaystyle\frac{1}{2}\int_{t-x}^{t+x}{H(\xi)}d\xi.
\end{equation}
Then by \eqref{eq:6.9}, the problem \eqref{eq:6.2} becomes
\begin{equation} \label{eq:6.11}
\begin{cases}
\tilde{y}_{tt}-\tilde{y}_{xx} = 0,\\
\tilde{f}(x) \triangleq\tilde{y}(0,x) = f(x)-\displaystyle\frac{1}{2}\int_{-x}^{x}{H(\xi)}d\xi,\\
\tilde{g}(x) \triangleq\tilde{y}(T,x) = g(x)-\displaystyle\frac{1}{2}\int_{T-x}^{T+x}{H(\xi)}d\xi,\\
\tilde{H}(t) \triangleq\tilde{y}_x(t,0) = 0,\\
\tilde{K}(t) \triangleq\tilde{y}_x(t,L) = 0.
\end{cases}
\end{equation}
And by \eqref{eq:6.4}, \eqref{eq:6.9}, the boundary data
$\tilde{f}$, $\tilde{g}$, $\tilde{H}$ and $\tilde{K}$ satisfy the
compatibility conditions
\begin{equation} \label{eq:6.12}
\begin{cases}
\tilde{f}'(0) = \tilde{f}'(L) = 0,\\
\tilde{g}'(0) = \tilde{g}'(L) = 0.
\end{cases}
\end{equation}
Therefore, we can make use of Theorem \ref{thm:5.1} (B) and obtain
the $C^2$ solution to the problem \eqref{eq:6.11}, and then the
$C^2$ solution to \eqref{eq:6.2}. This proves the part (B) in
Theorem \ref{thm:6.1}.

Thus, the proof of Theorem \ref{thm:6.1} is completed.
\end{proof}

\begin{Remark} \label{rem:6.1}
Theorem \ref{thm:6.1} shows that the wave equation defined on the
domain $[0,T]\times [0,L]$ still possesses the exact controllability
even in the cases of inhomogeneous Dirichlet boundary conditions and
inhomogeneous Neumann boundary conditions, provided that the
corresponding compatibility conditions are satisfied.
\end{Remark}

\section{Some nonlinear wave equations}\label{sec:7}

This section is devoted to the global exact controllability for some
nonlinear wave equations including a wave map equation arising from
geometry and the equations for the motion of relativistic strings in
the Minkowski space-time $\mathds{R}^{1+n}$.

\subsection{A wave map equation}

The theory of wave maps plays an important role in both mathematics
and theoretical physics. The wave map equation is highly
geometrical, and can be rewritten in many different ways. It is also
related to the Einstein equations in general relativity. In this
subsection, we consider the following TBVP for a kind of wave map
equation\footnote{In fact, the equation in \eqref{eq:7.1} is nothing
but the wave map from the Minkowski space $\mathds{R}^{1+1}$ to the
Riemmannian manifold $(\mathds{R},g)$, where $ds^2=gdy^2\triangleq
e^{-2y}dy^2$.}
\begin{equation} \label{eq:7.1}
\begin{cases}
y_{tt} - y_{xx} = y_t^2-y_x^2,\\
y(0,x) = f(x),\\
y(T,x) = g(x),
\end{cases}
\end{equation}
where $T$ is a given positive real number, and $f(x),g(x)$ are two
given $C^2$-smooth functions which stand for the initial and
terminal states, respectively.

By making the following transformation on the unknown
\begin{equation} \label{eq:7.2}
z(t,x) = e^{-y(t,x)},
\end{equation}
the TBVP \eqref{eq:7.1} can be rewritten as
\begin{equation} \label{eq:7.3}
\begin{cases}
z_{tt} - z_{xx} = 0,\\
z(0,x) = e^{-f(x)}\triangleq \hat{f}(x),\\
z(T,x) = e^{-g(x)}\triangleq \hat{g}(x),
\end{cases}
\end{equation}
Obviously, the TBVP \eqref{eq:7.1} has a $C^2$-smooth solution on
the strip $[0,T]\times\mathds{R}$ if and only if the TBVP
\eqref{eq:7.3} has a $C^2$-smooth positive solution on
$[0,T]\times\mathds{R}$. Therefore, in order to prove the existence
of a $C^2$-smooth solution of the TBVP \eqref{eq:7.1} on
$[0,T]\times\mathds{R}$, it suffices to show the existence of a
$C^2$-smooth positive solution of the TBVP \eqref{eq:7.3} on this
domain. In fact, if so, $y(t,x) = -\ln{z(t,x)}$ is a $C^2$-smooth
solution to the TBVP \eqref{eq:7.1}.

To do so, we suppose that
\begin{equation} \label{eq:7.4}
\inf{f(x)} > \sup{g(x)}.
\end{equation}
Similar to \eqref{eq:2.2} and \eqref{eq:2.3}, we introduce
\begin{equation} \label{eq:7.5}
\tilde{z}(t,x) = z(t,x) -
\displaystyle\frac{\hat{f}(x-t)+\hat{f}(x+t)}{2}
\end{equation}
and
\begin{align} \label{eq:7.6}
\tilde{f}(x) &= \hat{g}(x) -
\displaystyle\frac{\hat{f}(x-T)+\hat{f}(x+T)}{2} \nonumber
\\
&= e^{-g(x)} - \frac{1}{2}\left[e^{-f(x-T)}+e^{-f(x+T)}\right] > 0.
\end{align}
In \eqref{eq:7.6}, we have made use of the assumption
\eqref{eq:7.4}. Thus, the TBVP \eqref{eq:7.3} can be rewritten as
\begin{equation} \label{eq:7.7}
\begin{cases}
\tilde{z}_{tt} - \tilde{z}_{xx} = 0,\\
\tilde{z}(0,x) = 0,\\
\tilde{z}(T,x) = \tilde{f}(x).
\end{cases}
\end{equation}

As shown in Section \ref{sec:2}, we may construct a non-negative
$C^1$ function $u(x)$ defined on the interval $[-T, T]$ which
satisfies the condition \eqref{eq:2.8}, and define an initial
velocity $v(x)$ as shown in \eqref{eq:2.9}. By d'Alembert formula,
the solution of the Cauchy problem for the wave equation in
\eqref{eq:7.7} with the initial data
\begin{equation} \label{eq:7.8}
\tilde{z}(0,x) = 0,\quad \tilde{z}_t(0,x) = v(x)
\end{equation}
reads
\begin{equation} \label{eq:7.9}
\tilde{z}(t,x) =
\displaystyle\frac{1}{2}\int_{x-t}^{x+t}{v(\tau)}d\tau.
\end{equation}
It follows from \eqref{eq:7.9}, \eqref{eq:7.5} and the fact that
$\hat{f}(x)>0$ that the TBVP \eqref{eq:7.3} has a $C^2$-smooth
positive solution $z=z(t,x)$, provided that $v(x)\geq 0$.

By the above argument, the key point to show the existence of a
$C^2$-smooth positive solution of the TBVP \eqref{eq:7.3} on
$[0,T]\times\mathds{R}$ is to construct a non-negative initial
velocity $v(x)$. By \eqref{eq:2.9}, we have
\begin{Lemma} \label{lem:7.1}
Suppose that the function $\tilde{f}(x)$ defined by \eqref{eq:7.6}
satisfies
\begin{equation} \label{eq:7.10}
\begin{cases}
\displaystyle\sum_{i=1}^N\tilde{f}'\big(x-(2i-1)T\big)\geq 0, \quad \forall \;x> T,\\
\displaystyle\sum_{i=1}^N\tilde{f}'\big(x+(2i-1)T\big)\leq 0, \quad
\forall \;x<-T.
\end{cases}
\end{equation}
Then the function $v(x)$ defined by \eqref{eq:2.9} is non-negative.
\end{Lemma}

Lemma \ref{lem:7.1} is obvious, here we omit its proof.

Summarizing the above argument yields the following theorem.

\begin{Theorem} \label{thm:7.1}
Suppose that $f(x)$ and $g(x)$ are two $C^2$-smooth functions and
satisfy \eqref{eq:7.4} and \eqref{eq:7.10}. Then the TBVP
\eqref{eq:7.1} admits a $C^2$-smooth solution $y = y(t,x)$ defined
on the strip $[0,T]\times \mathds{R}$.
\end{Theorem}

In order to understand the condition \eqref{eq:7.10} clearly, we
present the following two examples.

\begin{Example} Choose the functions $f(x)$ and $g(x)$ such that the
function $\tilde{f}(x)$ defined by \eqref{eq:7.6} satisfies
\begin{equation} \label{eq:7.11}
\begin{cases}
\tilde{f}'(x)\geq 0, \quad \forall \;x> 0,\\
\tilde{f}'(x)\leq 0, \;\quad \forall \;x< 0.
\end{cases}
\end{equation}
In this case, it is easy to see that $\tilde{f}(x)$ satisfies the
assumption \eqref{eq:7.10}, and then the TBVP \eqref{eq:7.1} has a
$C^2$-smooth solution on the domain $[0,T]\times \mathds{R}$. For
example, we may choose $f(x)$ and $g(x)$ satisfying
$$\tilde{f}(x) = x^{2n} + c,$$ where
$\tilde{f}(x)$ is defined by \eqref{eq:7.6}, $c$ is a constant and
$n$ is a positive integer.
\end{Example}

\begin{Example} Choose the functions $f(x)$ and $g(x)$ such that the
function $\tilde{f}(x)$ defined by \eqref{eq:7.6} satisfies
\begin{equation} \label{eq:7.12}
\begin{cases}
\tilde{f}'(x)\geq 0, \quad \forall \; x\in [0, 2T], \\
\tilde{f}'(x) = -\tilde{f}'(x+2T).
\end{cases}
\end{equation}
In the present situation, we have
\begin{equation} \label{eq:7.13}
\displaystyle\sum_{i=1}^N\tilde{f}'\big(x-(2i-1)T\big)
\begin{cases}
\geq 0, \; if \; N \; is \; odd, \\
=0, \; if \; N  \; is  \; even,
\end{cases}\quad if\;\; x\geq 0
\end{equation}
and
\begin{equation} \label{eq:7.14}
\displaystyle\sum_{i=1}^N\tilde{f}'\big(x+(2i-1)T\big)
\begin{cases}
\leq 0, \; if \; N \; is \; odd, \\
=0, \; if \; N  \; is  \; even,
\end{cases}  if \;\; x < 0.
\end{equation}
This implies that $\tilde{f}(x)$ satisfies the assumption
\eqref{eq:7.10}, and then the TBVP \eqref{eq:7.1} has a $C^2$-smooth
solution on the domain $[0,T]\times \mathds{R}$. As an example, we
may choose $f(x)$ and $g(x)$ satisfying
$$\tilde{f}'(x) =
\displaystyle\frac{\pi}{2T}\sin{\frac{\pi x}{2T}},\quad
\text{i.e.,}\quad \tilde{f}(x) = -\displaystyle\cos{\frac{\pi
x}{2T}} + c,$$ where $\tilde{f}(x)$ is defined by \eqref{eq:7.6} and
$c$ is a constant.
\end{Example}

\subsection{The equations for the motion of relativistic strings in
the Minkowski space-time $\mathds{R}^{1+n}$}

Let $X=(t,x_1,\cdots,x_n)$ be a position vector of a point in the
$(1+n)$-dimensional Minkowski space $\mathbb{R}^{1+n}$. Consider the
motion of a relativistic string and let $x_i=x_i(t,\theta)\;
(i=1,\cdots,n)$ be the local equation of its world surface, where
$(t,\theta)$ are the the surface parameters. The equations governing
the motion of the string read (see \cite{kong4})
\begin{equation}\label{eq:7.2.1}
|x_{\theta}|^2x_{tt}-2\langle x_t,x_{\theta}\rangle
x_{t\theta}+(|x_t|^2-1)x_{\theta\theta}=0.
\end{equation}
The system \eqref{eq:7.2.1} contains $n$ nonlinear partial
differential equations of second order. These equations also
describe the extremal surfaces in the Minkowski space
$\mathbb{R}^{1+n}$. Kong et al \cite{kong4} considered the Cauchy
problem for the equations \eqref{eq:7.2.1} with the following
initial data
\begin{equation}\label{eq:7.2.2}
t=0:\;\;x=p(\theta),\quad x_t=q(\theta),
\end{equation}
where $p$ is a given $C^2$ vector-valued function and $q$ is a given
$C^1$ vector-valued function. The Cauchy problem
\eqref{eq:7.2.1}-\eqref{eq:7.2.2} describes the motion of a free
relativistic string in the Minkowski space $\mathbb{R}^{1+n}$ with
the initial position $p(\theta)$ and initial velocity (in the
classical sense) $q(\theta)$. In particular, when $p(\theta)$ and
$q(\theta)$ are periodic, the string under consideration is a closed
one. It has shown that the global smooth solution of the Cauchy
problem \eqref{eq:7.2.1}-\eqref{eq:7.2.2} exists and is unique (see
\cite{kong4}). On the other hand, it is well known that closed form
representations of solutions for partial differential equations are
very important and fundamental in both mathematical analysis and
physical understanding. Unfortunately, nonlinear partial
differential equations in general do not possess representations of
solutions in closed form. Surprisingly, in the paper \cite{kong3},
we discovered a general solution formula in closed form for the
nonlinear wave equations \eqref{eq:7.2.1}. By introducing a new
concept of {\it generalized time-periodic function}, we proved that,
if the initial data is periodic, then the smooth solution of the
Cauchy problem \eqref{eq:7.2.1}-\eqref{eq:7.2.2} is {\it generalized
time-periodic}, namely, the space-periodicity also implies the
time-periodicity. This fact yields an interesting physical
phenomenon: {\it the motion of closed strings is always generalized
time-periodic}.

However, up to now there is not any result on the TBVP for the
equations \eqref{eq:7.2.1}. In this subsection we investigate this
problem and prove the global exact controllability of the equations
\eqref{eq:7.2.1}.

By \cite{kong0}, under a very natural assumption which is a
necessary and sufficient condition guaranteeing the motion is
physical, there exists a globally diffeomorphic transformation of
variables
\begin{equation}\label{eq:7.2.3}
\tau = t,\quad \vartheta = \vartheta(t,\theta)
\end{equation}
such that the system \eqref{eq:7.2.1} become
\begin{equation}\label{eq:7.2.4}
\tilde{x}_{\tau\tau}-\tilde{x}_{\vartheta\vartheta}=0,
\end{equation}
where
\begin{equation}\label{eq:7.2.5}
\tilde{x}(\tau,\vartheta)=x\big(t(\tau,\vartheta),\theta(\tau,\vartheta)\big),
\end{equation}
in which $t=t(\tau,\vartheta),\;\theta=\theta(\tau,\vartheta)$ is
the inverse of the transformation \eqref{eq:7.2.3}.

By Theorem \ref{thm:2.1}, the system \eqref{eq:7.2.4} possesses the
global exact controllability. Noting the transformation
\eqref{eq:7.2.3} is globally diffeomorphic, we obtain the global
exact controllability of the system \eqref{eq:7.2.1}.

In physics, the global exact controllability of the system
\eqref{eq:7.2.1} implies some interesting physical phenomena. For
example, if we take the periodic initial data $p(\theta)$ which
stands for a closed string, and a non-periodic terminal condition
(e.g., $x(T,\theta)=\theta$) which denotes an open string, then the
global exact controllability of the system \eqref{eq:7.2.1} means
that a closed string may become an open string. This fact implies
that the topological structure of the string may change in its
motion process.

\section{Hyperbolic curvature flow and Gage-Hamilton's theorem} \label{sec:8}
Let $\gamma(t)$ be a one parameter family of closed convex smooth
curves in the plane. The position vector $X(t,s)$ parameterizes the
curve and the curvature is $k(t,s)$. The hyperbolic curvature flow
is described by the following wave equation
\begin{equation} \label{eq:8.1}
k_{tt} - k_{ss} = 0,
\end{equation}
where the subscript $\nu$ denotes partial differentiation with
respect $\nu$.

Consider the TBVP for the wave equation \eqref{eq:8.1} with the
following initial condition and terminal condition
\begin{equation} \label{eq:8.2}
k(0,s) = f(s) , \quad k(T,s)=k_0 >0,
\end{equation}
where $f(s)$ is a given non-negative periodic smooth function with
the period $L$, $T$ is a positive real number, and $k_0$ is a
positive constant. We have
\begin{Theorem} \label{thm:8.1}
There exists a positive constant $k_0$ such that the TBVP
\eqref{eq:8.1}-\eqref{eq:8.2} admits a non-negative smooth solution
$k=k(t,s)$, provided that $\displaystyle\frac{T}{L}$ is a rational
number with $\displaystyle\frac{2T}{L}\not\in \mathds{N}$.
\end{Theorem}

\begin{Remark} \label{rem:8.1}
By the basic theorem on plane curves, Theorem \ref{thm:8.1} implies
that, under the hyperbolic curvature flow, any closed convex smooth
curve in the plane can evolve into a circle in a finite time, and
remains convex for sequent times. This is a result analogous to the
well-know theorem of Gage and Hamilton for the mean curvature flow
of plane curves.
\end{Remark}

We now prove Theorem \ref{thm:8.1}.

\begin{proof} Consider the TBVP for the wave equation \eqref{eq:8.1} with the
following initial condition and terminal condition
\begin{equation} \label{eq:8.3}
k(0,s) = f(s) , \quad k(T,s)=k_*,
\end{equation}
where $k_*$ is an arbitrary given positive constant. By Theorem
\ref{thm:1-2.1}, the TBVP \eqref{eq:8.1}, \eqref{eq:8.3} has a
global $L$-periodic solution $k = k(t, s)$. Define
\begin{equation} \label{eq:8.4}
v(s) =  \partial_t k(0, s).
\end{equation}
Clearly, $v(s)$ is the initial velocity corresponding to the
solution $k=k(t,s)$. That is to say, the solution of the Cauchy
problem for the wave equation \eqref{eq:8.1} with the initial data
$k(0,s)=f(s),\;k_t(0,s)=v(s)$ is nothing but $k=k(t,s)$, moreover
this solution satisfies the terminal data $k(T,s)=k_*$.

Notice that $v(s)$ is $L$-periodic. Let
\begin{equation} \label{eq:8.5}
M = \min_{s\in[0,L]}{v(s)}.
\end{equation}
If $M\geq0$, then by D'Alembert formula
\begin{equation} \label{eq:8.6}
k(t,s)=\displaystyle\frac{f(s+t)+f(s-t)}{2}+\displaystyle\frac{1}{2}\int_{s-t}^{s+t}v(\tau)d\tau\geq0,\quad\forall
\;(t,s)\in [0,T]\times\mathds{R}.
\end{equation}
Taking $k_0=k_*$, we finish the proof of Theorem \ref{thm:8.1}.

Otherwise, it holds that $v(s)+|M|\geq 0$. Thus, it is easy to see
that $\bar{k}\triangleq k(t,s)+|M|t$ is a global $L$-periodic smooth
solution of the TBVP
\begin{equation} \label{eq:8.7}
\begin{cases}
\bar{k}_{tt}(t, s) - \bar{k}_{ss}(t, x) = 0,\\
\bar{k}(0, s) = f(s),\\
\bar{k}(T, s) = k_* + |M|T.
\end{cases}
\end{equation}
Obviously, the initial velocity corresponding to the solution
$\bar{k}(t,s)$ reads
\begin{equation} \label{eq:8.8}
\bar{v}(s)\triangleq\partial_t \bar{k}(0, s)=v(s)+|M| \geq 0.
\end{equation}
This implies that the solution is non-negative, i.e., $\bar{k}(t,
s)\geq 0$ for all $(t,s)\in [0,T]\times\mathds{R}$. Taking $k_0=k_*+
|M|T$ gives the conclusion of Theorem \ref{thm:8.1}. Thus, the proof
of Theorem \ref{thm:8.1} is completed.
\end{proof}

\section{Summary and discussions} \label{sec:9}
In the present paper we introduce three new concepts for
second-order hyperbolic equations, they read two-point boundary
value problem, global exact controllability and exact
controllability, respectively. These second-order hyperbolic
equations considered here include many important partial
differential equations arising from both theoretical aspects and
applied fields, e.g., mechanics (fluid mechanics and elasticity),
physics, engineering, control theory and geometry, etc., the typical
examples are wave equation, hyperbolic Monge-Amp\`{e}re equation,
wave map. For several kinds of important linear and nonlinear wave
equations, in this paper we prove the existence of smooth solutions
of the two-point boundary value problems and show the global exact
controllability of these wave equations. In particular, we
investigate the two-point boundary value problem for one-dimensional
wave equation defined on a closed curve and prove the existence of
smooth solution, this implies the exact controllability of this kind
of wave equation. Furthermore, based on this, we study the two-point
boundary value problems for the wave equation defined on a strip
with Dirichlet or Neumann boundary conditions and show that the
equation still possesses the exact controllability in these cases.
Finally, as an application of Theorem \ref{thm:1-2.1}, we introduce
the hyperbolic curvature flow and prove a result analogous to the
well-known theorem of Gage and Hamilton \cite{gh} for the curvature
flow of plane curves. This result can be viewed as a simple
application of hyperbolic partial differential equation to both
geometry and topology.

Usually, ``two-point boundary value problem" has another meaning: it
applies to a boundary problem for a second order ODE in a bounded
interval. The present paper generalizes this concept to the case of
second-order hyperbolic partial differential equations. The
two-point boundary value problem for partial differential equations
is a new research topic. According to the authors' knowledge, up to
now, few of results on the two-point boundary value problems for
partial differential equations, in particular, hyperbolic partial
differential equations (even for linear or nonlinear wave equations)
have been known. Therefore, the present paper can be viewed as the
first work in this new research topic.

It is well-known that, there are a lot of deep and beautiful results
on the two-point boundary value problems for ordinary differential
equations and for second-order differential inclusions. On the other
hand, there are many important results on the {\it boundary} control
problems for wave equations and hyperbolic systems. However, the
two-point boundary value problems and the boundary control problems
are essentially different two kinds of problems. Both of them play
an important role in both theoretical and applied aspects.

The main aim of this paper is to introduce the concept ``two-point
boundary value problem" for second-order hyperbolic equations and to
show the global exact controllability or exact controllability for
several kinds of important linear and nonlinear wave equations.
Although the results obtained in this paper are restrictive in some
sense (they only apply to cases in which the solution to the Cauchy
problem can be found explicitly: linear (classical) wave equations
and their reformulations), these results shed further light on the
study of this new research topic. In the future we will investigate
the following open problems which seem to us more interesting and
important: (i) what happens if we consider a general equation of the
form $$y_{tt}-y_{xx} = f(y)$$ with a regular and bounded function
$f(y)$? (ii) what happens if we include regular coefficients and/or
lower order terms? (iii) what happens if we consider a quasilinear
wave equation of the form $$y_{tt}-(\mathscr{C}(y_{x}))_x = 0$$ with
a smooth and increasing function $\mathscr{C}(\nu)$, e.g.,
$\mathscr{C}(\nu)=\frac{\nu}{\sqrt{1+\nu^2}}$? (iv) what happens if
we consider other models, particularly some nonlinear models arising
from applied fields such as control theory, fluid dynamics,
elasticity as well as engineering, etc.

\vskip 4mm

\noindent{{\bf Acknowledgements.}} This work was supported in part
by the NNSF of China (Grant No. 10971190) and the Qiu-Shi Chair
Professor Fellowship from Zhejiang University.

\end{document}